\documentclass[a4paper,11pt]{article}
\usepackage[top=2cm, bottom=2cm, left=2cm, right=2cm]{geometry}\usepackage{ascmac}
\usepackage[dvipdfmx]{graphicx}
\usepackage{array,booktabs,float}
\usepackage{amsmath,amssymb}
\usepackage{amsthm}
\usepackage{mathtools}
\usepackage{tikz-cd}
\usepackage{here}
\usepackage{hhline}
\usepackage{indentfirst}
\usepackage{url}
\theoremstyle{definition}
\newcommand{\ackname}{Acknowledgement}
\newenvironment{acknowledgement}{
\section*{\ackname\hfill\hfill}}{\relax}
\newtheorem{definition}{Definition}[section]
\newtheorem{theorem}[definition]{Theorem}

\newtheorem{lemma}[definition]{Lemma}

\newcommand{\m}{\mathfrak{m}}
\newcommand{\GL}{\mathop{\mathrm{GL}}\nolimits}
\newcommand{\Sym}{\mathop{\mathrm{Sym}}\nolimits}
\newcommand{\Proj}{\mathbb{P}}
\newcommand{\disc}{\mathrm{disc}}
\newcommand{\sli}{\mathrm{sli}}
\newcommand{\Gr}{\mathop{\mathrm{Gr}}\nolimits}
\newcommand{\PP}{\mathop{\mathrm{Proj}}\nolimits}

\newcounter{num}
\newcommand{\Rnum}[1]{\setcounter{num}{#1}\Roman{num}}
\newcommand{\rnum}[1]{\setcounter{num}{#1}\roman{num}}
\DeclareFontFamily{U}{matha}{\hyphenchar\font45}
\DeclareFontShape{U}{matha}{m}{n}{
    <5> <6> <7> <8> <9> <10> gen * matha
    <10.95> matha10 <12> <14.4> <17.28> <20.74> <24.88> matha12
}{}
\DeclareSymbolFont{matha}{U}{matha}{m}{n}
\DeclareMathSymbol{\divides}{3}{matha}{"17}
\DeclareMathSymbol{\notdivides}{3}{matha}{"1F}
\title{Isomorphism types of commutative algebras of rank 7 over 
an arbitrary algebraically closed field of characteristic not 2 or 3}
\author{Naoto ONDA}
\begin{document}
	\maketitle
\begin{abstract}
We classify isomorphism types of unital commutative algebras of rank 7 
over an algebraically closed field of characteristic not 2 or 3 completely. 
\end{abstract}

\section{Introduction}
Let $k$ be an algebraically closed field of characteristic not 2 or 3. 
We say that a $k$-algebra is of rank $n$ if its dimension as a $k$-vector space is $n$. 
In this paper, any $k$-algebra is assumed to be commutative and have a unit. 
In \cite{Poonen07isomorphismtypes}, Bjorn Poonen classified isomorphic types of $k$-algebras of rank up to 6 
and showed that the number of isomorphism types of $k$-algebras of rank $n$ is finite 
if and only if $n\leq 6$. 
In this paper, we complete a classification of $k$-algebras of rank 7. 
We can easily see that any $k$-algebra of finite rank can be decomposed into 
a product of Artinian local $k$-algebras and 
its decomposition is unique up to permutations, 
so we have only to classify Artinian local $k$-algebras of rank 7. 
Let $(A, \m)$ be an Artinian local $k$-algebra. 
The number $d_i = \dim_k(\m^i/\m^{i + 1})$ depends only on the isomorphism type of $A$. 
We define the vector $\vec{d}$ to be $(d_1, d_2, \dots)$. 
In \cite{Yu2015TowardsTC}, Alexandria Yu has classified 
isomorphism types of Artinian local $k$-algebras of rank 7 except the type $\vec{d} = (3, 3)$, 
so we have only to classify $k$-algebras of the remaining type. 
Let $\GL(3)$ be the group of $3\times 3$ regular matrices over $k$, 
$\Sym_2(3)$ be the space of symmetric $3\times 3$ matrices over $k$ 
and $\Gr_3(V)$ be the set of 3-dimensional subspaces of a $k$-vector space $V$. 
%For an appropriate action of $\GL(3)$ on $\Gr_3(\Sym_2(3))$, 
There is a one-to-one correspondence 
between the isomorphism types of Artinian local $k$-algebras of type $\vec{d} = (3, 3)$ 
and the orbits under the natural action of $\GL(3)$ on $\Gr_3(\Sym_2(3))$.
For a more general statement, see Lemma \ref{d1d2}.
%Therefore, we classify the orbits and classification of singular orbits, introduced below, is Main Theorem \ref{main}. 
%In this paper, we classify singular orbits 
%because nonsingular orbits has already been classified in \cite[Theorem 5.8]{MR3472915}. 

To describe our classification of orbits, we define the discriminant scheme of an orbit. 
Let $[W]$ be the orbit including an element $W$ of $\Gr_3(\Sym_2(3))$. 
We define the discriminant scheme $\disc[W]$ of $[W]$ to be the closed subscheme of $\PP{k[x,y,z]} = \Proj^2$ 
defined by the homogeneous polynomial
\[
	\det(A_1\cdot x + A_2\cdot y + A_3\cdot z),
\]
where $\{A_1, A_2, A_3\}$ is a basis of $W$. It is well-defined up to projective transformations. 
We say that $[W]$ is nonsingular (resp.\ singular) if $\disc[W]$ is (resp.\ is not) a nonsingular cubic curve. 
Nonsingular orbits are already classified in \cite[Theorem 5.8]{MR3472915}. 
In this paper, we classify the singular orbits. 
Our main result is as follows.
\begin{theorem}
	All schemes which appear as discriminant schemes of singular orbits are as follows.
	\begin{description}
		\item[(\Rnum{1})] The projective plane.
		\item[(\Rnum{2})] A line with multiplicity three.
		\item[(\Rnum{3})] The union of a line with multiplicity two and another line.
		\item[(\Rnum{4})] The union of distinct three lines whose intersection is not one point.
		\item[(\Rnum{5})] The union of a quadratic curve and a line whose intersection is distinct two points.
		\item[(\Rnum{6})] The union of a quadratic curve and a tangent line.
		\item[(\Rnum{7})] An irreducible cubic curve with a cusp.
		\item[(\Rnum{8})] An irreducible cubic curve with a node.
	\end{description}
%	The number of singular orbits is ten. 
	There are two orbits each of whose discriminant scheme is the case (\Rnum{1}) 
	(resp.\ (\Rnum{4})). 
%	There do not exist orbits each of whose discriminant scheme is three lines which intersect at one point. 
	There is only one orbit for each of the other cases.
\end{theorem}\noindent
See Theorem $\ref{main}$ for detail. 

In order to distinguish two orbits in the case (\Rnum{1}),  
we consider the ``slices'' of them, 
which appeared in the construction of the bijection of \cite[Theorem 5.8]{MR3472915}. 
If $[W]$ is nonsingular, $\disc[W]$ and $\sli[W]$ are isomorphic; 
it was proved in the construction of the bijection of \cite[Theorem 5.8]{MR3472915}. 
However, they are not always isomorphic if $[W]$ is singular. 
%We prove that the pair $(\disc[W], \sli[W])$ characterizes the orbit. 
To show every orbit appears in our list, 
we use classification of the orbits of 2-dimensional subspace of a $k$-vector space $\Sym_2(3)$, 
due to \cite{Poonen07isomorphismtypes}.

\section{Main Theorem}\label{secmaintheorem}
Let $k$ be an algebraically closed field of characteristic not 2 or 3. 
In this paper, every $k$-algebra is assumed to be commutative and have a unit. 
We classify Artinian local $k$-algebras of rank 7. Let $(A, \m)$ be an Artinian local $k$-algebra. 
We define $\vec{d} = (d_1, d_2, \dots)$, where $d_i = \dim_k(\m^i/\m^{i + 1})$. 
Let $\GL(n)$ be the group of $n\times n$ regular matrices over $k$, 
$\Sym_2(n)$ be the space of symmetric $n\times n$ matrices over $k$ 
and $\Gr_n(V)$ be the set of $n$-dimensional subspaces of a $k$-vector space $V$. 
We define an action of $\GL(m)$ on $\Gr_n(\Sym_2(m))$ by 
	\[
 	   \begin{array}{c@{\,\,}c@{\,\,}c}
	       \GL(m)\times\Gr_n(\Sym_2(m))&\longrightarrow&\Gr_n(\Sym_2(m)),\\
	       \rotatebox{90}{$\in$}&&\rotatebox{90}{$\in$}\\
	       (M, W)&\longmapsto&M^\top\cdot W\cdot M
	    \end{array}
	\]
	where $M^\top$ is the transpose of $M$. 
Here is an important lemma.

\begin{lemma}[{\cite[Lemma 1.2 (3)]{Poonen07isomorphismtypes}}]\label{d1d2}
	There is a canonical bijection 
	between isomorphism types of Artinian $k$-algebras of the type $\vec{d} = (d_1, d_2)$ 
	and orbits of the above action on $\Gr_{d_2}(\Sym_2(d_1))$. 
\end{lemma}

In \cite{Yu2015TowardsTC}, Alexandria Yu has classified isomorphism types of 
Artinian local $k$-algebras of rank 7 except the type $\vec{d} = (3, 3)$, 
so we have only to classify the orbits of the action on $\Gr_3(\Sym_2(3))$. 
Let $[W]$ denote the orbit including an element $W$ of $\Gr_3(\Sym_2(3))$. 
We define the discriminant scheme of $[W]$ to be the closed subscheme $\disc[W]$ of $\PP{k[x,y,z]} = \Proj^2$ 
defined by the homogeneous polynomial
\[
	\det(A_1\cdot x + A_2\cdot y + A_3\cdot z),
\]
where $\{A_1, A_2, A_3\}$ is a basis of $W$. 
It is well-defined up to projective transformations. 
We say that $[W]$ is nonsingular (resp.\ singular) 
if $\disc[W]$ is (resp.\ is not) a nonsingular cubic curve. 
According to \cite[Theorem 5.8]{MR3472915}, there is a bijection 
between nonsingular orbits and isomorphism classes of triples $(C, \mathfrak{L}, P)$, 
where $C$ is a nonsingular cubic curve, 
$\mathfrak{L}$ is a degree 3 line bundle over $C$ 
and $P$ is a 2-torsion point of ${\mathop{\mathrm{Pic}}\nolimits}^0(C)$. 
Therefore, it suffices to classify the singular orbits. 
In fact, we have the following theorem.

\begin{theorem}\label{main}
	The number of singular orbits is ten. We list a basis of a representative of each orbit in Table \ref{3table}.
	\begin{table}[htb]
		\centering
		\begin{tabular}{c|c|c}
			disc$[W]$ & a representative of $[W]$ & the corresponding $k$-algebra $k[x,y,z]/I$ \\ \hhline{ =|=|= }
			& & \\[-10pt]
			\Rnum{1}& $\left(\begin{bmatrix}
				0 & 0 & 0\\
				0 & 0 & 0\\
				0 & 0 & 1
			\end{bmatrix}, 
			\begin{bmatrix}
				0 & 0 & 0\\
				0 & 1 & 0\\
				0 & 0 & 0
			\end{bmatrix}, 
			\begin{bmatrix}
				0 & 0 & 0\\
				0 & 0 & 1\\
				0 & 1 & 0
			\end{bmatrix}\right)$ & $(x^2, xy, xz, y^3, y^2z, yz^2, z^3)$ \\ [20pt]
			 & $\left(\begin{bmatrix}
				0 & 0 & 1\\
				0 & 0 & 0\\
				1 & 0 & 0
			\end{bmatrix}, 
			\begin{bmatrix}
				0 & 0 & 0\\
				0 & 0 & 1\\
				0 & 1 & 0
			\end{bmatrix}, 
			\begin{bmatrix}
				0 & 0 & 0\\
				0 & 0 & 0\\
				0 & 0 & 1
			\end{bmatrix}\right)$ & $(x^2, xy, y^2, z^3)$ \\ [20pt]\hline
			& & \\[-10pt]
			\Rnum{2}& $\left(\begin{bmatrix}
				1 & 0 & 0\\
				0 & 0 & 0\\
				0 & 0 & 0
			\end{bmatrix}, 
			\begin{bmatrix}
				0 & 0 & 1\\
				0 & 1 & 0\\
				1 & 0 & 0
			\end{bmatrix}, 
			\begin{bmatrix}
				0 & 1 & 0\\
				1 & 0 & 0\\
				0 & 0 & 0
			\end{bmatrix}\right)$ & $(xz - y^2, yz, z^2, x^3, x^2y, x^2z)$ \\ [20pt]\hline
			& & \\[-10pt]
			\Rnum{3}& $\left(\begin{bmatrix}
				1 & 0 & 0\\
				0 & 0 & 0\\
				0 & 0 & 0
			\end{bmatrix}, 
			\begin{bmatrix}
				0 & 0 & 1\\
				0 & 1 & 0\\
				1 & 0 & 0
			\end{bmatrix}, 
			\begin{bmatrix}
				0 & 0 & 0\\
				0 & 1 & 0\\
				0 & 0 & 0
			\end{bmatrix}\right)$ & $(xy, yz, z^2, x^3, x^2z, y^3)$ \\ [20pt]\hline
			& & \\[-10pt]
			\Rnum{4}& $\left(\begin{bmatrix}
				1 & 0 & 0\\
				0 & 0 & 0\\
				0 & 0 & 0
			\end{bmatrix}, 
			\begin{bmatrix}
				0 & 0 & 0\\
				0 & 1 & 0\\
				0 & 0 & 0
			\end{bmatrix}, 
			\begin{bmatrix}
				0 & 0 & 0\\
				0 & 0 & 0\\
				0 & 0 & 1
			\end{bmatrix}\right)$ & $(xy, xz, yz, x^3, y^3, z^3)$\\ [20pt]
			 & $\left(\begin{bmatrix}
				0 & 0 & 0\\
				0 & 1 & 0\\
				0 & 0 & -1
			\end{bmatrix}, 
			\begin{bmatrix}
				1 & 0 & 0\\
				0 & 0 & 0\\
				0 & 0 & -1
			\end{bmatrix}, 
			\begin{bmatrix}
				2 & 1 & 1\\
				1 & 0 & 1\\
				1 & 1 & 0
			\end{bmatrix}\right)$ & $(y^2 + z^2 - x^2 + xz, xy - xz, xz - yz, x^2y, xz^2)$ \\ [20pt]\hline
			& & \\[-10pt]
			\Rnum{5} & $\left(\begin{bmatrix}
				1 & 0 & 0\\
				0 & 0 & 0\\
				0 & 0 & 0
			\end{bmatrix}, 
			\begin{bmatrix}
				0 & 0 & 1\\
				0 & 1 & 0\\
				1 & 0 & 0
			\end{bmatrix}, 
			\begin{bmatrix}
				0 & 0 & 0\\
				0 & 0 & 0\\
				0 & 0 & 1
			\end{bmatrix}\right)$ & $(xy, y^2 - xz, yz, x^3, z^3)$ \\ [20pt]\hline
			& & \\[-10pt]
			\Rnum{6} & $\left(\begin{bmatrix}
				0 & 0 & 0\\
				0 & 1 & 0\\
				0 & 0 & 1
			\end{bmatrix}, 
			\begin{bmatrix}
				1 & 0 & 0\\
				0 & 0 & 0\\
				0 & 0 & 1
			\end{bmatrix}, 
			\begin{bmatrix}
				2 & 1 & 0\\
				1 & 0 & 0\\
				0 & 0 & 0
			\end{bmatrix}\right)$ & $(xz, yz, x^2 - 2xy + y^2 - z^2, x^3, y^3)$ \\ [20pt]\hline
			& & \\[-10pt]
			\Rnum{7} & $\left(\begin{bmatrix}\label{R7}
				1 & 0 & 0\\
				0 & 0 & 0\\
				0 & 0 & 0
			\end{bmatrix}, 
			\begin{bmatrix}
				0 & 0 & 1\\
				0 & 1 & 0\\
				1 & 0 & 0
			\end{bmatrix}, 
			\begin{bmatrix}
				0 & 0 & 0\\
				0 & 0 & 1\\
				0 & 1 & 0
			\end{bmatrix}\right)$ & $(xy, y^2 - xz, z^2, x^3)$ \\ [20pt]\hline
			& & \\[-10pt]
			\Rnum{8} & $\left(\begin{bmatrix}
				1 & 0 & 0\\
				0 & 0 & 0\\
				0 & 0 & 0
			\end{bmatrix}, 
			\begin{bmatrix}
				0 & 0 & 1\\
				0 & 1 & 0\\
				1 & 0 & 0
			\end{bmatrix}, 
			\begin{bmatrix}
				0 & 1 & 0\\
				1 & 0 & 0\\
				0 & 0 & 1
			\end{bmatrix}\right)$ & $(xy - z^2, y^2 - xz, yz, x^3)$
		\end{tabular}
		\caption{All orbits of $\GL(3)$-action to $\Gr_3(\Sym_2(3))$}\label{3table}
	\end{table}
	Corresponding discriminant schemes are as follows.
	\begin{description}
		\item[(\Rnum{1})] The projective plane.
		\item[(\Rnum{2})] A line with multiplicity three.
		\item[(\Rnum{3})] The union of a line with multiplicity two and another line.
		\item[(\Rnum{4})] The union of distinct three lines whose intersection is not one point.
		\item[(\Rnum{5})] The union of a quadratic curve and a line whose intersection is distinct two points.
		\item[(\Rnum{6})] The union of a quadratic curve and a tangent line.
		\item[(\Rnum{7})] An irreducible cubic curve with a cusp.
		\item[(\Rnum{8})] An irreducible cubic curve with a node.
	\end{description}
\end{theorem}

To prove that they are different orbits with one another, we define slices of them.

\begin{definition}[\cite{MR3472915}]
	Suppose that $W$ is spanned by three $3\times 3$ matrices $(a^i_{jk})_{j,k}$ $(i=1,2,3)$. 
	We define the closed subscheme sli$[W]$ of $\PP{k[x,y,z]}$ by the equation 
	\[
		\det\left[\begin{bmatrix}
				a^1_{11} & a^2_{11} & a^3_{11}\\
				a^1_{21} & a^2_{21} & a^3_{21}\\
				a^1_{31} & a^2_{31} & a^3_{31}
			\end{bmatrix}x + \begin{bmatrix}
				a^1_{12} & a^2_{12} & a^3_{12}\\
				a^1_{22} & a^2_{22} & a^3_{22}\\
				a^1_{32} & a^2_{32} & a^3_{32}
			\end{bmatrix}y + \begin{bmatrix}
				a^1_{13} & a^2_{13} & a^3_{13}\\
				a^1_{23} & a^2_{23} & a^3_{23}\\
				a^1_{33} & a^2_{33} & a^3_{33}
			\end{bmatrix}z
		\right].
	\]
	We can easily see that sli$[W]$ is well-defined up to projective transformations. 
	We call it the slice of $[W]$. 
\end{definition}

To show that any orbit is one of the orbits listed in Table \ref{3table}, 
we use the following lemma.

\begin{lemma}[{\cite{Poonen07isomorphismtypes}}]\label{2subsp}
	All orbits of $\Gr_2(\Sym_2(3))$ under the action of $\GL(3)$ are in Table \ref{2table}. 
	The discriminant scheme of each orbit, which is a closed subscheme of $\Proj^1$ 
	and is well-defined up to projective transformations, is defined in the same way as above. 
	\begin{table}[htb]
		\centering
		\begin{tabular}{c|c}
		disc$[W]$ & a representative of $[W]$\\ \hhline{ = | =}
			& \\[-10pt]
		$\Proj^1$ & $\left(\begin{bmatrix}
				0 & 0 & 0\\
				0 & 0 & 0\\
				0 & 0 & 1
			\end{bmatrix}, 
			\begin{bmatrix}
				0 & 0 & 0\\
				0 & 1 & 0\\
				0 & 0 & 0
			\end{bmatrix}\right)$ \\[20pt]
		 & $\left(\begin{bmatrix}
				0 & 0 & 0\\
				0 & 0 & 0\\
				0 & 0 & 1
			\end{bmatrix}, 
			\begin{bmatrix}
				0 & 0 & 0\\
				0 & 0 & 1\\
				0 & 1 & 0
			\end{bmatrix}\right)$ \\[20pt]
		 & $\left(\begin{bmatrix}
				0 & 0 & 1\\
				0 & 0 & 0\\
				1 & 0 & 0
			\end{bmatrix}, 
			\begin{bmatrix}
				0 & 0 & 0\\
				0 & 0 & 1\\
				0 & 1 & 0
			\end{bmatrix}\right)$\\ [20pt]\hline
			& \\[-10pt]
		($1^3$) & $\left(\begin{bmatrix}%1
				1 & 0 & 0\\
				0 & 0 & 0\\
				0 & 0 & 0
			\end{bmatrix}, 
			\begin{bmatrix}
				0 & 0 & 1\\
				0 & 1 & 0\\
				1 & 0 & 0
			\end{bmatrix}\right)$ \\[20pt] \hline
			& \\[-10pt]
		($1^2 1$) & $\left(\begin{bmatrix}%1
				1 & 0 & 0\\
				0 & -1 & 0\\
				0 & 0 & 0
			\end{bmatrix}, 
			\begin{bmatrix}
				1 & 0 & 1\\
				0 & 0 & 1\\
				1 & 1 & 0
			\end{bmatrix}\right)$ \\[20pt]
			 & $\left(\begin{bmatrix}%1
				0 & 0 & 0\\
				0 & 0 & 1\\
				0 & 1 & 0
			\end{bmatrix}, 
			\begin{bmatrix}
				1 & 0 & 0\\
				0 & 0 & 0\\
				0 & 0 & 0
			\end{bmatrix}\right)$\\[20pt]
			 & $\left(\begin{bmatrix}
				0 & 0 & 0\\
				0 & 1 & 0\\
				0 & 0 & 1
			\end{bmatrix}, 
			\begin{bmatrix}
				0 & 0 & 1\\
				0 & 0 & 0\\
				1 & 0 & 1
			\end{bmatrix}\right)$\\[20pt] \hline
			& \\[-10pt]
			(111) & $\left(\begin{bmatrix}
				0 & 0 & 0\\
				0 & 1 & 0\\
				0 & 0 & 1
			\end{bmatrix}, 
			\begin{bmatrix}
				1 & 0 & 0\\
				0 & 0 & 0\\
				0 & 0 & 1
			\end{bmatrix}\right)$ 
		\end{tabular}
		\caption{All orbits of $\GL(3)$-action to $\Gr_2(\Sym_2(3))$}\label{2table}
	\end{table}
\end{lemma}

For example, let $W$ be an element of $\Gr_3(\Sym_2(3))$ 
whose discriminant scheme is a cubic curve with a cusp. 
There exists a line $H$ in $\Proj^2$ such that the intersection of $H$ and $\disc[W]$ 
is one point with multiplicity three. 
By changing a basis of $W$, which corresponds to a projective transformation of $\PP k[x,y,z]$, 
we may assume that the line $H$ is defined by $z = 0$. 
Therefore, the space $W$ contains a 2-dimensional subspace $U$ such that $U$ belongs to the orbit 
of case $(1^3)$ in Table \ref{2table}. 
Moving $W$ by an appropriate $3\times3$ matrix, we may assume that $U$ is spanned by 
\[
	\begin{bmatrix}
				1 & 0 & 0\\
				0 & 0 & 0\\
				0 & 0 & 0
			\end{bmatrix}\text{ and }
			\begin{bmatrix}
				0 & 0 & 1\\
				0 & 1 & 0\\
				1 & 0 & 0
			\end{bmatrix}
\]
and thus $W$ contains them.

\begin{proof}[Proof of Theorem $\ref{main}$]
	First we show that the orbits listed in Table \ref{3table} are mutually different.  
	While the slice of the first case of (\Rnum{1}) in Theorem \ref{main} is also $\Proj^2$, 
	that of the second is a line with multiplicity three, 
	so we see that the first is not the same as the second. 
	While the first case of (\Rnum{4}) contains a rank one matrix, the second does not. 
	Therefore, the orbits listed in Table \ref{3table} are mutually different. 
	
	Next we show that every singular orbit is equal to one of the orbits listed in Table \ref{3table}. 
	Suppose that $[W]$ is an orbit and $W$ is a representative.
	\begin{description}
		\item[(\rnum{1})] Suppose that disc$[W]$ is the projective plane. 
			We may assume that $W$ contains a 2-dimensional subspace in the case $(\Proj^1)$ in Table \ref{2table}. 
			Thus we can take 
			\begin{align}
				&\left(\begin{bmatrix}  \label{1-1}
					0 & 0 & 0\\
					0 & 0 & 0\\
					0 & 0 & 1
				\end{bmatrix}, 
				\begin{bmatrix}
					0 & 0 & 0\\
					0 & 1 & 0\\
					0 & 0 & 0
				\end{bmatrix}, 
				\begin{bmatrix}
					a & b & c\\
					b & 0 & d\\
					c & d & 0
				\end{bmatrix}\right),\\
				&\left(\begin{bmatrix} \notag
					0 & 0 & 0\\
					0 & 0 & 0\\
					0 & 0 & 1
				\end{bmatrix}, 
				\begin{bmatrix}
					0 & 0 & 0\\
					0 & 0 & 1\\
					0 & 1 & 0
				\end{bmatrix}, 
				\begin{bmatrix}
					a & b & c\\
					b & d & 0\\
					c & 0 & 0
				\end{bmatrix}\right),\\
				\text{or} &\left(\begin{bmatrix} \notag
					0 & 0 & 1\\
					0 & 0 & 0\\
					1 & 0 & 0
				\end{bmatrix}, 
				\begin{bmatrix}
					0 & 0 & 0\\
					0 & 0 & 1\\
					0 & 1 & 0
				\end{bmatrix}, 
				\begin{bmatrix}
					a & b & 0\\
					b & c & 0\\
					0 & 0 & d
				\end{bmatrix}\right)
			\end{align}
			as a basis of $W$, where $(a,b,c,d)\neq (0,0,0,0)$. In the case (\ref{1-1}), it must satisfy $axyz - b^2xz^2 - cyz^2 = 0$, 
			so we may assume $(a,b,c,d) = (0,0,0,1)$. This is the first case of (\Rnum{1}) in Table \ref{3table}. 
			We can treat the other cases in the same way.
%			In case (\ref{1-2}), we may assume $(a,b,c,d) = (0,0,0,1), (0,0,1,0)$ in the same way.  
%			They are cases of (\Rnum{1}) in Table \ref{3table}. 
%			In case (\ref{1-3}), we may assume $(a,b,c,d) = (0,0,0,1)$.  
%			This is the second case of (\Rnum{1}) of Table \ref{3table}.
		\item[(\rnum{2})] Suppose that $\disc[W]$ is a line with multiplicity three. 
			Since there exists a line $H$ in $\Proj^2$ 
			such that the intersection of $H$ and $\disc[W]$ is a point with multiplicity three, 
			we may assume that $W$ contains a 2-dimensional subspace of the case $(1^3)$ in Table \ref{2table}. 
			Thus we can take 
			\[
				\left(\begin{bmatrix}
					1 & 0 & 0\\
					0 & 0 & 0\\
					0 & 0 & 0
				\end{bmatrix}, 
				\begin{bmatrix}
					0 & 0 & 1\\
					0 & 1 & 0\\
					1 & 0 & 0
				\end{bmatrix}, 
				\begin{bmatrix}
					0 & a & 0\\
					a & b & c\\
					0 & c & d
				\end{bmatrix}\right)
			\]
			as a basis of $W$, where $(a,b,c,d) \neq (0,0,0,0)$.  
			Let $G = z((bd - c^2)z + dy)$ and $J = - a^2dz^3 + 2acyz^2 - by^2z - y^3$. 
			Since the defining polynomial of $\disc[W]$ is 
			$F = xG + J$ 
			and $F$ is reducible, $G$ and $J$ have a common factor if $G$ is not 0. 
			If $G$ is not 0, we have  
			$J = ((bd - c^2)z + dy)K$ for some quadratic polynomial $K$ in $k[y,z]$ 
			whose coefficient of $y^2$ is nonzero 
			and $F = ((bd - c^2)z + dy)(xz + K)$. This contradicts the assumption that 
			$F$ can be written as a product of three linear polynomials. 
			Thus $G$ must be 0 and we have $d = c = 0$. 
			Therefore, we have $F = - by^2z - y^3 = L^3$ for some linear polynomial $L$ 
			and we may assume $(a,b,c,d) = (1,0,0,0)$. This is the case (\Rnum{2}) in Table \ref{3table}. 
		\item[(\rnum{3})] Suppose that $\disc[W]$ is the union of a line with multiplicity two and another line. 
			As in the case (\rnum{2}), we may assume that $W$ is spanned by 
			\[
				\left(\begin{bmatrix}
					1 & 0 & 0\\
					0 & 0 & 0\\
					0 & 0 & 0
				\end{bmatrix}, 
				\begin{bmatrix}
					0 & 0 & 1\\
					0 & 1 & 0\\
					1 & 0 & 0
				\end{bmatrix}, 
				\begin{bmatrix}
					0 & a & 0\\
					a & b & 0\\
					0 & 0 & 0
				\end{bmatrix}\right),
			\]
			where $(a,b) \neq (0,0)$ and the defining polynomial of $\disc[W]$ is $F = - by^2z - y^3 = H^2L$ 
			for some distinct linear polynomials $H$ and $L$. Thus we have $b \neq 0$ and we may take 			\[
				\left(\begin{bmatrix}
					1 & 0 & 0\\
					0 & 0 & 0\\
					0 & 0 & 0
				\end{bmatrix}, 
				\begin{bmatrix}
					0 & 0 & 1\\
					0 & 1 & 0\\
					1 & 0 & 0
				\end{bmatrix}, 
				\begin{bmatrix}
					0 & a & 0\\
					a & 1 & 0\\
					0 & 0 & 0
				\end{bmatrix}\right)
			\]
			as a basis of $W$. Moving $W$ by the action of 
			\[
				\begin{bmatrix}
					1 & 0 & 0\\
					-a & 1 & 0\\
					0 & a & 1
				\end{bmatrix},
			\]
			we may assume $a = 0$. This is the case (\Rnum{3}) in Table \ref{3table}.
		\item[(\rnum{4})] Assume that $\disc[W]$ is the union of distinct three lines whose intersection is not one point. 
			Since there exists a line $H$ in $\Proj^2$ 
			such that the intersection of $H$ and $\disc[W]$ consists of distinct three points,  
			we may assume that $W$ contains a 2-dimensional subspace of the case $(111)$ in Table \ref{2table}. 
			Therefore, we can take 
			\[
				\left(\begin{bmatrix}
					0 & 0 & 0\\
					0 & 1 & 0\\
					0 & 0 & 1
				\end{bmatrix}, 
				\begin{bmatrix}
					1 & 0 & 0\\
					0 & 0 & 0\\
					0 & 0 & 1
				\end{bmatrix}, 
				\begin{bmatrix}
					a & b & c\\
					b & 0 & d\\
					c & d & 0
				\end{bmatrix}\right)
			\]
			as a basis of $W$. 
			Let $G = x^2 + axz  - (b^2 + d^2)z^2$ and $J = z((2bcd - ad^2)z^2 - (b^2 + c^2)xz + ax^2)$. 
			Then the defining polynomial of $\disc[W]$ is $F = xy^2 + Gy + J$. 
			Since $F$ is reducible, we have $F = (xy + K)(y + L)$ 
			for some quadratic polynomial $K$ and some linear polynomial $L$ in $k[x,z]$. 
			Now that $F$ is a product of three linear polynomials, $K$ must be divided by $x$. 
			Therefore, $G$ and $J$ must be divided by $x$. 
			Thus we see that $b^2 + d^2 = 0$ and $d(2bc - ad) = 0$. 
			\begin{itemize}
				\item Suppose $d = 0$. Then we have $b = 0$ and $F = x(y^2 + (x + az)y + z(-c^2z + ax))$. 
					Since $F$ is a product of three linear polynomials, there exists $r \in k^{\times}$ such that
					\[
						rz - \frac{c^2}{r}z + \frac{a}{r}x = x + az.
					\]
					Therefore, $c = 0$. This is the first case of (\Rnum{4}) in Table \ref{3table}.
				\item Suppose $d \neq 0$. Then we have $2bc - ad = 0$, $b^2 + d^2 = 0$ and 
					$F = x(y^2 + (x + az)y + z(- (b^2 + c^2)z + ax))$. Since $F$ is a product 
					of three linear polynomials, 
					there exists $r \in k^{\times}$ such that
					\[
						rz + \frac{- (b^2 +c^2)}{r}z + \frac{a}{r}x = x +az.
					\]
					Therefore, we see that $b^2 + c^2 = 0$. Thus we can take 
					\begin{align}
						&\left(\begin{bmatrix}
							0 & 0 & 0\\
							0 & 1 & 0\\
							0 & 0 & 1
						\end{bmatrix}, 
						\begin{bmatrix}
							1 & 0 & 0\\
							0 & 0 & 0\\
							0 & 0 & 1
						\end{bmatrix}, 
						\begin{bmatrix}
							2i & i & 1\\
							i & 0 & 1\\
							1 & 1 & 0
						\end{bmatrix}\right), \notag \\
						&\left(\begin{bmatrix}
							0 & 0 & 0\\
							0 & 1 & 0\\
							0 & 0 & 1
						\end{bmatrix}, 
						\begin{bmatrix}
							1 & 0 & 0\\
							0 & 0 & 0\\
							0 & 0 & 1
						\end{bmatrix}, 
						\begin{bmatrix}
							2i & -i & 1\\
							-i & 0 & -1\\
							1 & -1 & 0
						\end{bmatrix}\right), \label{-i}\\
						&\left(\begin{bmatrix}
							0 & 0 & 0\\
							0 & 1 & 0\\
							0 & 0 & 1
						\end{bmatrix}, 
						\begin{bmatrix}
							1 & 0 & 0\\
							0 & 0 & 0\\
							0 & 0 & 1
						\end{bmatrix}, 
						\begin{bmatrix}
							2i & i & -1\\
							i & 0 & -1\\
							-1 & -1 & 0
						\end{bmatrix}\right), \notag \\ \text{or }
						&\left(\begin{bmatrix}
							0 & 0 & 0\\
							0 & 1 & 0\\
							0 & 0 & 1
						\end{bmatrix}, 
						\begin{bmatrix}
							1 & 0 & 0\\
							0 & 0 & 0\\
							0 & 0 & 1
						\end{bmatrix}, 
						\begin{bmatrix}
							2i & -i & -1\\
							-i & 0 & 1\\
							-1 & 1 & 0
						\end{bmatrix}\right)\notag
					\end{align}
					as a basis of $W$, where $i$ is a root of $t^2 + 1$. 
					In fact, all of them are in the second orbit of 
					(\Rnum{4}) in Table \ref{3table}. For example, if we move (\ref{-i}) by
					\[
						\begin{bmatrix}%1
							1 & 0 & 0\\
							0 & -1 & 0\\
							0 & 0 & i
						\end{bmatrix},
					\]
					it turns into the second representative of (\Rnum{4}) in Table \ref{3table}. 
			\end{itemize}
		\item[(\rnum{5})] Suppose that disc$[W]$ is the union of a quadratic curve and a line 
		whose intersection consists of distinct two points. 
			As in the case (\rnum{2}), we can take 
			\[
				\left(\begin{bmatrix}
					1 & 0 & 0\\
					0 & 0 & 0\\
					0 & 0 & 0
				\end{bmatrix}, 
				\begin{bmatrix}
					0 & 0 & 1\\
					0 & 1 & 0\\
					1 & 0 & 0
				\end{bmatrix}, 
				\begin{bmatrix}
					0 & a & 0\\
					a & b & c\\
					0 & c & d
				\end{bmatrix}\right)
			\]
			as a basis of $W$, where $(a,b,c,d) \neq (0,0,0,0)$. 
			Let $G = dy + (bd - c^2)z$ and $J = y^3 + by^2z - 2acyz^2 + a^2dz^3$. Then 
			the defining polynomial of $\disc[W]$ is $F = xzG - J$. 
			Since $F$ is reducible, we have $zG$ and $J$ have a common factor if $G$ is not 0.
			If the polynomial $G$ were 0, the discriminant scheme became a union of three lines.
			Thus we may assume that $J$ is divided by $G$. 
			Therefore, we may assume $d = 1$ and $a = c(c^2 - b)$. Moving $W$ by the action of 
			\[
				\begin{bmatrix}
					0 & 0 & 1\\
					0 & 1 & c\\
					1 & -c & b - c^2
				\end{bmatrix},
			\]
			we can verify that this space turns into the representative of (\Rnum{5}) in Table \ref{3table}.
		\item[(\rnum{6})] Assume that disc$[W]$ is the union of a quadratic curve and a tangent line. 
			As in the case (\rnum{4}), 
			we can take 
			\[
				\left(\begin{bmatrix}
					0 & 0 & 0\\
					0 & 1 & 0\\
					0 & 0 & 1
				\end{bmatrix}, 
				\begin{bmatrix}
					1 & 0 & 0\\
					0 & 0 & 0\\
					0 & 0 & 1
				\end{bmatrix}, 
				\begin{bmatrix}
					a & b & c\\
					b & 0 & d\\
					c & d & 0
				\end{bmatrix}\right)
			\]
			as a basis of $W$, where $(a,b,c,d) \neq (0,0,0,0)$. 
			Let $G = y + az$, $J = (b^2 + d^2)y + d(ad - 2bc)z$ and $K = y^2 + ayz - (b^2 + c^2)z^2$. 
			Then the defining polynomial of $\disc[W]$ is $F = Gx^2 + Kx - Jz^2$. 
			Since $F$ must be divided by a linear polynomial, 
			$J$ and $K$ are divided by $G$ or there exists $r \in k^{\times}$ such that $GJ/r - rz^2 = K$. 
			\begin{itemize}
				\item Suppose that $J$ and $K$ are divided by $G$. Then we see that
					\begin{align}
						0 &= b^2 + c^2 \label{bc}\\
						0 &= c(ac - 2bd) \label{ac-2bd}\\
						F &= (y + az)(x^2 + xy - (b^2 + d^2)z^2)\notag.
					\end{align}
					Since a line $y + az = 0$ must be tangent to $x^2 + xy - (b^2 + d^2)z^2 = 0$, we have 
					\begin{align}
						a^2 + 4(b^2 + d^2) = 0. \label{abd}
					\end{align}
					From (\ref{ac-2bd}), we have $a^2c^4 = 4b^2c^2d^2$. 
					Moreover, we have $c^4(a^2 + 4d^2)  = 0$ from (\ref{bc}). 
					From (\ref{abd}), we see that $bc = 0$. Therefore, we obtain $b = c = 0$ from (\ref{bc}) 
					and $a^2  + 4d^2 = 0$ from (\ref{abd}). 
					Thus $W$ is equal to 
					\[
						\left(\begin{bmatrix}
							0 & 0 & 0\\
							0 & 1 & 0\\
							0 & 0 & 1
						\end{bmatrix}, 
						\begin{bmatrix}
							1 & 0 & 0\\
							0 & 0 & 0\\
							0 & 0 & 1
						\end{bmatrix}, 
						\begin{bmatrix}
							2 & 0 & 0\\
							0 & 0 & i\\
							0 & i & 0
						\end{bmatrix}\right),
					\]
					where $i$ is a root of $t^2 + 1$. Moving $W$ by the action of 
					\[
						\begin{bmatrix}
							0 & 0 & i\\
							0 & 1 & 0\\
							-i & 0 & 0
						\end{bmatrix},
					\]
					it turns into the representative of (\Rnum{6}) in Table \ref{3table}.	
				\item Suppose that there exists $r \in k^{\times}$ such that $GJ/r - rz^2 = K$. Then we see that
					\begin{align}
						r &= b^2 + d^2 \notag \\ 
						0 &= d(ad - 2bc) \label{ad-2bc}\\
						0 &= c^2 - d^2 \label{cd}\\
						F &= ((y + az)x - (b^2 + d^2)z^2)(x + y).\notag
					\end{align}
					Since $x + y = 0$ must be tangent to $(y + az)x - (b^2 + d^2)z^2 = 0$, we have
					\begin{align}
						a^2 = 4(b^2 + d^2). \label{adb}
					\end{align}
					From (\ref{ad-2bc}), we have $a^2d^4 = 4b^2c^2d^2$. 
					Moreover, we have $d^4(a^2 - 4b^2)  = 0$ from (\ref{cd}). 
					From (\ref{adb}), we see that $d = 0$. Therefore, we obtain $d = c = 0$ from (\ref{cd}) 
					and $a^2  - 4b^2 = 0$ from (\ref{adb}). 
					Thus $W$ is equal to either
					\[
						\left(\begin{bmatrix}
							0 & 0 & 0\\
							0 & 1 & 0\\
							0 & 0 & 1
						\end{bmatrix}, 
						\begin{bmatrix}
							1 & 0 & 0\\
							0 & 0 & 0\\
							0 & 0 & 1
						\end{bmatrix}, 
						\begin{bmatrix}
							2 & 1 & 0\\
							1 & 0 & 0\\
							0 & 0 & 0
						\end{bmatrix}\right)
					\]
					or 
					\[
						\left(\begin{bmatrix}
							0 & 0 & 0\\
							0 & 1 & 0\\
							0 & 0 & 1
						\end{bmatrix}, 
						\begin{bmatrix}
							1 & 0 & 0\\
							0 & 0 & 0\\
							0 & 0 & 1
						\end{bmatrix}, 
						\begin{bmatrix}
							2 & -1 & 0\\
							-1 & 0 & 0\\
							0 & 0 & 0
						\end{bmatrix}\right).
					\] Moving the second by the action of 
					\[
						\begin{bmatrix}
							1 & 0 & 0\\
							0 & -1 & 0\\
							0 & 0 & 1
						\end{bmatrix},
					\]
					it turns into the representative of (\Rnum{6}) in Table \ref{3table}.	
			\end{itemize}
		\item[(\rnum{7})] Suppose that $\disc[W]$ is a cubic curve with a cusp. 
			As in the case (\rnum{2}), 
			we can take 
			\[
				\left(\begin{bmatrix}
					1 & 0 & 0\\
					0 & 0 & 0\\
					0 & 0 & 0
				\end{bmatrix}, 
				\begin{bmatrix}
					0 & 0 & 1\\
					0 & 1 & 0\\
					1 & 0 & 0
				\end{bmatrix}, 
				\begin{bmatrix}
					0 & a & 0\\
					a & b & c\\
					0 & c & d
				\end{bmatrix}\right)
			\]
			as a basis of $W$, where $(a,b,c,d) \neq (0,0,0,0)$. 
			Then the defining polynomial of $\disc[W]$ is 
			$F = - y^3 - b^2y^2z + 2acyz^2 + dxyz + (bd - c^2)xz^2 - a^2dz^3$. 
			We can verify that $[1 : 0 : 0]$ is a singular point, so this is the cusp 
			and we see that $d$ must be zero and $c \neq 0$. We may assume $c = 1$. 
			Thus $W$ is equal to 
			\[
				\left(\begin{bmatrix}
					1 & 0 & 0\\
					0 & 0 & 0\\
					0 & 0 & 0
				\end{bmatrix}, 
				\begin{bmatrix}
					0 & 0 & 1\\
					0 & 1 & 0\\
					1 & 0 & 0
				\end{bmatrix}, 
				\begin{bmatrix}
					0 & a & 0\\
					a & b & 1\\
					0 & 1& 0
				\end{bmatrix}\right).
			\]
			Moving $W$ by the action of 
			\[
				\begin{bmatrix}
					1 & 0 & 0\\
					\frac{b}{3} & 1 & 0\\
					\frac{-2b^2}{9} - a & \frac{ - b}{3} & 1
				\end{bmatrix},
			\]
			we see that it is in the orbit of (\Rnum{7}) in Table \ref{3table}.
		\item[(\rnum{8})] Suppose that $\disc[W]$ is a cubic curve with a node. 
			As in the case (\rnum{2}), 
			we can take 
			\[
				\left(\begin{bmatrix}
					1 & 0 & 0\\
					0 & 0 & 0\\
					0 & 0 & 0
				\end{bmatrix}, 
				\begin{bmatrix}
					0 & 0 & 1\\
					0 & 1 & 0\\
					1 & 0 & 0
				\end{bmatrix}, 
				\begin{bmatrix}
					0 & a & 0\\
					a & b & c\\
					0 & c & d
				\end{bmatrix}\right)
			\]
			as a basis of $W$, where $(a,b,c,d) \neq (0,0,0,0)$. 
			Then the defining polynomial $F$ of $\disc[W]$ 
			is $- y^3 - b^2y^2z + 2acyz^2 + dxyz + (bd - c^2)xz^2 - a^2dz^3$. 
			We can verify that $[1 : 0 : 0]$ is a singular point, so this is the node 
			and we see that $d \neq 0$. We may assume $d = 1$. Thus $W$ is equal to 
			\[
				\left(\begin{bmatrix}
					1 & 0 & 0\\
					0 & 0 & 0\\
					0 & 0 & 0
				\end{bmatrix}, 
				\begin{bmatrix}
					0 & 0 & 1\\
					0 & 1 & 0\\
					1 & 0 & 0
				\end{bmatrix}, 
				\begin{bmatrix}
					0 & a & 0\\
					a & b & c\\
					0 & c & 1
				\end{bmatrix}\right).
			\]
			Moving $W$ by the action of 
			\[
				\begin{bmatrix}
					1 & 0 & 0\\
					c & 1 & 0\\
					b - 2c^2 & - c & 1
				\end{bmatrix},
			\]
			we may assume $b = c = 0$. Moving $W$ by the action of 
			\[
				\begin{bmatrix}
					1 & 0 & 0\\
					0 & 0 & t\\
					0 & t^2 & 0
				\end{bmatrix}
			\]
			for some $t\in k$ such that $t^3 = a$, we may assume $a = 0$. 
			This is the case (\Rnum{8}) in Table \ref{3table}.
		\item[(\rnum{9})] Suppose that $\disc[W]$ is the union of three lines which intersect at one point. 
			As in the case (\rnum{2}), 
			we can take 
			\[
				\left(\begin{bmatrix}%1
					1 & 0 & 0\\
					0 & 0 & 0\\
					0 & 0 & 0
				\end{bmatrix}, 
				\begin{bmatrix}
					0 & 0 & 1\\
					0 & 1 & 0\\
					1 & 0 & 0
				\end{bmatrix}, 
				\begin{bmatrix}
					0 & a & 0\\
					a & b & 0\\
					0 & 0 & 0
				\end{bmatrix}\right)
			\]
			as a basis of $W$, where $(a,b) \neq (0,0)$.  Then the defining polynomial of $\disc[W]$ is 
			$F = - by^2z - y^3$. This contradicts the assumption that $F$ must be written as a product of three distinct 
			linear polynomials.\qedhere
 	\end{description}
\end{proof}
\begin{acknowledgement}
The author is very grateful to his advisor Yoichi Mieda for several helpful suggestions and warmful encouragements during this work.
\end{acknowledgement}
\bibliographystyle{my_amsalpha}
\bibliography{reference}

\end{document}